\documentclass[letterpaper]{amsart}
\usepackage{tikz-cd}
\usepackage{amsmath,amssymb,amsthm,mathrsfs}
\usepackage[utf8]{inputenc}
\usepackage{hyperref}
\theoremstyle{plain} 
\newtheorem{theorem}{Theorem}
\newtheorem*{theorem*}{Theorem}
\newtheorem{prop}[theorem]{Proposition}
\newtheorem{lemma}[theorem]{Lemma}
\newtheorem{coro}[theorem]{Corollary}

\newtheorem{claim}{Claim} 
\theoremstyle{definition} \newtheorem{definition}{Definition}
\theoremstyle{remark} \newtheorem{remark}{Remark}
\theoremstyle{definition} \newtheorem*{ack}{Acknowledgment}
\author{Chuanhao Wei}
\title[log-Kodaira dimension and zeros of holomorphic log-one-forms]{Logarithmic Kodaira dimension and zeros of holomorphic log-one-forms}
\date{}
\begin{document}
\maketitle

\section{Introduction}
In this paper, we prove the conjecture that appeared in \cite{W16}, which is a natural generalization of the work of M. Popa and C. Schnell, \cite{PS14}.
\begin{theorem}\label{T:log general}
The zero-locus of any global holomorphic log-one-form on a projective
log-smooth pair $\left(X,D\right)$ of log-general type must be non-empty.
\end{theorem}
Actually, we can prove a more general result:
\begin{theorem}\label{T:log pair}
Let $\left(X,D\right)$ be a projective log-smooth pair, and let 
$$W \subset H^0 \left(X,\Omega^1_X\left(\log D\right)\right)$$ 
be a linear subspace that consists of global holomorphic log-one-forms with empty zero-locus. 
 Then the dimension of $W$ can be at most $\dim X-\kappa\left(X, D\right)$, where $\kappa
$ stands for the logarithmic Kodaira dimension.
\end{theorem}

This is a corollary to the following main theorem of this paper. Before stating the theorem, we recall the notations and some propositions about qusi-abelian varieties in the sense of Iitaka, given in \cite{W16}. See also \cite{Ii76} and \cite{Fu14} for more details.

\begin{definition}
$T^{r,d}$ is a quasi-abelian variety (in the sense of Iitaka), if it is an extension of a $d$-dimensional abelian variety $A^d$ by an algebraic torus $\mathbb{G}_m^{r}$, i.e. it is a connected commutative algebraic group which has the following Chevalley decomposition
$$1\to \mathbb{G}_m^r\to T^{r,d}\to A^d\to 1.$$
In particular, $T^{r,d}$ is a principal $\mathbb{G}_m^r$-bundle over $A^d$.
\end{definition}
Consider the following group homomorphism:
$$\rho: \mathbb{G}^r_m\to PGL(r,\mathbb{C}),$$
given by
$$\rho(\lambda_1,...,\lambda_r)=
\begin{bmatrix}
    1 &   &     &0    \\
      &\lambda_1 &   &   \\
      &   & \ddots &   \\
       0&    &  & \lambda_r
\end{bmatrix}
$$
Let $P^{r,d}:=T^{r,d}\times_{\rho}\mathbb{P}^r=T^{r,d}\times\mathbb{P}^r/\mathbb{G}^r_m$, which is a $\mathbb{P}^r$-bundle over $A^d$. We can view $P^{r,d}$ as a compactification of $T^{r,d}$ which naturally carries the $T^{r,d}$ action on it and denote the boundary divisor $L$, which is of simple normal crossing and is $T^{r,d}$-invariant for each stratum.

We say that $\left(X,D\right)$ is a log-smooth pair if $X$ is a smooth variety and $D$ is a reduced divisor on $X$ with normal crossing support. Given two log smooth pairs $\left(X,D^X\right)$ and $\left(Y,D^Y\right)$, and a morphism $f:X\to Y$, we say that $f$ is a morphism of log-pairs and denoted by $f: \left(X,D^X\right)\to \left(Y,D^Y\right)$, if 
$$f^{-1}D^Y:=\text{Supp}\left(f^*D^Y\right)\subset D^X.$$

Given a holomorphic log-one-form $\theta$ on a log smooth pair $\left(X,D\right)$, we use $Z\left(\theta\right)$ to denote the zero-locus of $\theta$ as a global section of the locally free sheaf $\Omega^1_X\left(\log  D\right)$.
\begin{theorem}\label{T:over quasi-abelian variety}
Let $\left(X,D\right)$ be a projective log-smooth pair, $\left(P^{r,d},L\right)$ be the canonical $\mathbb{P}^r$-bundle compactification of a quasi-abelian variety $T^{r,d}$, with the boundary divisor $L$. Denote by $p:P^{r,d}\to A^d$, the natural projection. Given a morphism of log-pairs $f:\left(X,D\right)\to \left(P^{r,d},L\right)$, if there is a positive integer $k$ and an ample line bundle $\mathcal{A}$ on $A^d$, such that 
$$H^0\left(X, \left(\omega_X\left(D\right)\right)^{\otimes k}\otimes f^*\left(p^*\mathcal{A}^{-1}\otimes \mathcal{O}_{P^{r,d}}\left(-L\right)\right)\right)\neq 0,$$
then $Z\left(f^*\omega\right)\neq \emptyset$, for any $\omega\in H^0\left(\Omega^1_{P^{r,d}}\left(\log   L\right)\right).$

Further, for generic such $\omega$, we have $Z\left(f^*\omega\right)\cap\left(X\setminus f^{-1}\left(L\right)\right)\neq \emptyset$.
\end{theorem}
\begin{remark}\label{R:holds for any r}
Actually, we only need to show the case for $r=0$ or $1$. We refer to the proof of Theorem \ref{T:over P^{r,d}} below for the details.
\end{remark}

\begin{remark}\label{R:generic}
The two statements in the previous theorem are actually equivalent to each other. In particular, it is a problem that only depends on the morphism over $T^{r,d}$. See Step 0 of the proof of this theorem in Section 4 for details.
\end{remark}

There are two important applications of Theorem \ref{T:log pair}. One of them is that we get an affirmative answer (in a much more general setting) to a question posed by F. Catanese and M. Schneider, Corollary \ref{C:log smooth} and Corollary \ref{C:log smooth*}. Another application is that we give an answer to the algebraic hyperbolicity part of Shafarevich's conjecture, with the generic fiber being klt and of log-general type, Corollary \ref{C:fiber of lg}. E. Viehweg and K. Zuo have shown the previous two questions in the case that the base is one dimensional and the general fibers are projective and smooth, \cite{VZ01}.

In Section 2, we first prove Theorem \ref{T:log pair} assuming Theorem \ref{T:over quasi-abelian variety}. Then we show the geometric applications mentioned above. In Section 3, we give a simplified proof of the main theorem in \cite{PS14} without applying the generic vanishing of mixed Hodge modules on abelian varieties, \cite{PS13}. In Section 4, assuming Claim \ref{C:main claim} (Main Claim), we prove Theorem \ref{T:over quasi-abelian variety}. In Section 5, we first recall some results about logarithmic comparison in mixed Hodge modules in \cite{W17a}. Then we show a vanishing result that will be used to prove the Main Claim. In Section 6, we relate the Main Claim with the theory of mixed Hodge modules and prove it.

All the \emph{varieties} that appear in this paper are assumed to be reduced but possibly reducible separated schemes of finite type over the field of complex number $\mathbb{C}$. We use strict right $\mathscr{D}$-modules to represent \emph{mixed Hodge modules}, forgetting the weight filtration. All mixed Hodge modules in this paper are assumed to be algebraic. In particular, they are assumed to be extendable and polarizable, \cite{Sa90}.

\begin{ack}
The author would like to express his gratitude to his advisor Christopher Hacon for suggesting this topic and useful discussions. The author thanks Honglu Fan, Kalle Karu, Mihnea Popa, Christian Schnell, Lei Wu and Ziwen Zhu for answering his questions, and especially Schnell for suggesting a better choice of Hodge module which is used in the paper that essentially simplifies the proof of the main theorem. 

During the preparation of this paper, the author was partially supported by DMS-1300750, DMS-1265285 and a grant from the Simons Foundation, Award Number 256202.
\end{ack}

\section{Proof of Theorem \ref{T:log pair} and applications}
We start by showing that Theorem \ref{T:over quasi-abelian variety} implies the following
\begin{theorem}\label{T:over P^{r,d}}
Given a morphism of log smooth pairs $f:\left(X,D\right)\to \left(P^{r,d}, L\right)$, where $\left(P^{r,d},L\right)$ is the canonical $\mathbb{P}^r$-bundle compactification of a quasi-abelian variety $T^{r,d}$, with the boundary divisor $L$,
then there exists a linear subspace $W\subset H^0\left(P^{r,d},\Omega^1_{P^{r,d}}\left(\log   L\right)\right)$ of co-dimension $\dim X-\kappa\left(X, D\right)$, such that $Z\left(f^*\omega\right)\neq \emptyset$, for any $\omega\in W$.
\end{theorem}

\begin{proof}
Denote $U=X\setminus D$. The statement is vacuous when $\kappa\left(X,D\right)=-\infty$, so we can assume that $\kappa\left(X,D\right)\geq 0.$ 
Let $\mu:\left(X', D'\right)\to \left(X,D\right)$ be a birational modification, such that $g:X'\to Z$ is a smooth model for the Iitaka fibration associated to the log-canonical bundle $\omega_X\left(D\right)$. It is not hard to see that, if $\left(X'_z, D'_z\right)$ is a very general fiber over $Z$, then it is a log-smooth pair of dimension $\delta\left(U\right)=\dim U-\kappa\left(X,D\right)$ with log-Kodaira dimension $0$. Hence due to \cite[Theorem 28]{Ka81}, the quasi-albanese map of $X'_z\setminus D'_z$ is an open algebraic fiber space, and by \cite[Theorem 27]{Ka81}, it is not hard to see that its image in $T^{r,d}_U$ is a quasi-abelian variety. Further, by Proposition \ref{P:coutable subgroup} below, we know that there are at most countably many quasi-abelian sub-algebraic group in $T^{r,d}_U$. Hence, the image in $T^{r,d}_U$ of every fiber of $g$ is a dense subset of a translate of a single quasi-abelian sub-algebraic group. Letting $T$ be the quotient of $T^{r,d}_U$ by that sub quasi-abelian sub-algebraic group, we have $\dim T\geq m-\delta\left(U\right)$, where $m=r+d=\dim T^{r,d}_U$. Hence we get an induced rational map $Z\dashrightarrow T$. Let $\left(P,L\right)$ be the projective-space-bundle compactification of $T$. Though $P^{r,d}\dashrightarrow P$ is just a rational map in general, after a toroidal log-resolution $\left(\hat{P}^{r,d},\hat{L}\right)\to \left(P^{r,d}, L\right)$, the induced map of log-pairs $\left(\hat{P}^{r,d},\hat{L}\right)\to \left(P,L\right)$ is a morphism. After making a toroidal log resolution 
$$\tau:\left(\hat{X}, \hat{D}\right)\to\left(X,D\right),$$
we can get a morphism of log-pairs $\left(\hat{X}, \hat{D}\right)\to \left(\hat{P}^{r,d},\hat{L}\right)$ (\cite{AK00} Section 1, Remark 1.4). Composing it with the previous one, we get a morphism of log-pairs
$$\hat{f}:\left(\hat{X}, \hat{D}\right)\to \left(P,L\right).$$
Consider the following commutative diagram:
\begin{center}
\begin{tikzcd}
\left(X',D'\right)\arrow[dashed]{r}\arrow{d}
&\left(\hat{X}, \hat{D}\right)\arrow{d}{\hat{f}}\\
Z\arrow[dashed]{r}
&\left(P,L\right).
\end{tikzcd}
\end{center}
Now for any line bundle $\hat{\mathcal{A}}$ on $P$, we have 
$$H^0\left(\hat{X}, \left(\omega_{\hat{X}}\left(\hat{D}\right)\right)^{\otimes k}\otimes \hat{f}^*\hat{\mathcal{A}} \right)\neq 0,$$
for some integer $k$. Since $\tau:\left(\hat{X}, \hat{D}\right)\to\left(X,D\right)$ is a toroidal log-resolution, we have 
that $\theta\in H^0\left(\hat{X},\Omega_{X}^1\left(\log   D\right)\right)$ has zero-locus if and only if $\tau^*\theta \in H^0\left(\hat{X},\Omega_{\hat{X}}^1\left(\log   \hat{D}\right)\right)$ has zero-locus. Hence now we only need to show that, for any $\omega\in H^0\left(P, \Omega_P^1\left(\log  L\right)\right)$, $\hat{f}^*\omega\in H^0\left(\hat{X},\Omega_{\hat{X}}^1\left(\log   \hat{D}\right)\right)$ will have zero-locus. Now we can conclude the proof by  Theorem \ref{T:over quasi-abelian variety}.

For the rest of the proof, we show that we actually just need the case that $r=0$ or $1$ in Theorem \ref{T:over quasi-abelian variety}. Denote $\left(P,L\right)$ by $\left(P^{r,d}, L\right)$ again to specify the dimension of the base quasi-abelian variety and projective space fiber, and $T^{r,d}:=P^{r,d}\setminus L$. If $r\geq 2$, we will argue that the statement can be reduced to the $r=1$ case. Consider the subgroup $\mathbb{G}_{m}^r$ of $T^{r,d}$, which is the fiber over $0$, the unit of $A^d$. Fix $\{x_1,...,x_r\}$, a global algebraic coordinate system on $\mathbb{G}_{m}^r$, and $x_1=...=x_r=1$ gives the unit $e$ of $\mathbb{G}_{m}^r$. Given a vector $\mathbf{a}:=[a_1,...,a_r]\in \mathbb{Z}^r$, assuming $\mathbf{a}\neq \mathbf{0}$, let's consider the divisor $G_\mathbf{a}$ of $\mathbb{G}_{m}^r$ defined by $x_1^{a_1}\cdot...\cdot x_r^{a_r}=1$. It is evident that it is actually a subgroup of $\mathbb{G}_{m}^r$, hence of $T^{r,d}$. Let's denote the algebraic quotient by
$$q_{\mathbf{a}}:T^{r,d}\to T^{r,d}/G_\mathbf{a}.$$
It is evident that $T^{r,d}/G_\mathbf{a}\simeq T^{1,d}_\mathbf{a}$, for some quasi-abelian variety $T^{1,d}_\mathbf{a}$ with indicated dimensions. Denote by $\left(P^{1,d}_\mathbf{a}, L\right)$, the canonical $\mathbb{P}^1$-bundle compactification of $T^{1,d}_\mathbf{a}$. As in the argument above, replacing $\left(\hat{X}, \hat{D}\right)$ and $\left(P^{r,d}, L\right)$ by a toroidal log-resolution, we obtain morphisms of log-pairs:
\begin{align*}
q_{\mathbf{a}}: \left(P^{r,d}, L\right)&\to \left(P^{1,d}_\mathbf{a}, L\right)\\
g_\mathbf{a}: \left(\hat{X}, \hat{D}\right)&\to \left(P^{1,d}_\mathbf{a}, L\right)
\end{align*}
Denote
$$W_\mathbf{a}=q_{\mathbf{a}}^* H^0\left(P^{1,d}_\mathbf{a}, \Omega^1_{P^{1,d}_\mathbf{a}}(\log L)\right)\subset H^0\left(P^{r,d}, \Omega^1_{P^{r,d}}(\log L)\right).$$
Note that $\cup_\mathbf{a}W_\mathbf{a}$ is a dense subset of $H^0\left(P^{r,d}, \Omega^1_{P^{r,d}}(\log L)\right)$. This is because $\{dx_1/x_1,...,dx_r/x_r\}$ gives a basis of $H^0\left(\Omega^1_{P^{r,d}}(\log L)\right)/p^*H^0\left(A^d, \Omega^1_{A^d}\right)$, and it is straight forward to check that $W_\mathbf{a}$ is the co-set 
$$\mathbb{C}\cdot \left(a_1dx_1/x_1+...+a_rdx_r/x_r\right)+p^*H^0\left(A^d, \Omega^1_{A^d}\right).$$
Since $\hat{f}^*\omega$ has no zero-locus is an open condition on $\omega\in H^0\left(P^{r,d}, \Omega^1_{P^{r,d}}(\log L)\right)$, now it suffices to show the statement for each $g_\mathbf{a}: \left(\hat{X}, \hat{D}\right)\to \left(P^{1,d}_\mathbf{a}, L\right)$, which is implied by Theorem \ref{T:over quasi-abelian variety}.
\end{proof}
\begin{remark}\label{R:toric compactification}
For a fixed quasi-abelian variety $T^{r,d}$, fixing any smooth projective $r$-dimensional toric variety $P^r$, we have the canonical $P^{r}$-bundle compactification of $T^{r,d}$, with a natural simply normal crossing boundary divisor $L$, and we also denote it by $\left(P^{r,d},L\right)$. All the $\left(P^{r,d}, L\right)$ that appear in this paper can be viewed in this setting.
\end{remark}

\begin{prop}\label{P:coutable subgroup}
There are at most countably many quasi-abelian sub-algebraic groups in a quasi-abelian variety.
\end{prop}
\begin{proof}
It is evident that we only need to show the case for any abelian variety and for any algebraic torus. For the abelian variety case, it is obvious by noting that there are at most countably many sub-lattice in a given lattice of finite dimension. For the algebraic torus case, it is not hard to see that we only need to show that there are at most countably many algebraic group auto-morphism $\mathbb{G}^1_m\to \mathbb{G}^1_m$. Fix a global coordinator $t$, the auto-morphism will be given by the map of function $t\mapsto P[t,t^{-1}]$, where $P[t,t^{-1}]$ is a two-variable polynomial. Since it needs to be a group morphism, we have 
$$P\left[t^k,t^{-k}\right]=\left(P\left[t,t^{-1}\right]\right)^k,$$
for any $k\in \mathbb{Z}$. Hence it is not hard to conclude that $P[t,t^{-1}]=t^m$ for some $m\in \mathbb{Z}$, which has countably many choices.
\end{proof}

Recall that given a smooth quasi-projective variety $U$ and a log-smooth compactification $\left(X,D\right)$ of $U$, we have that $T_1\left(U\right):=H^0\left(X, \Omega_X^1\left(\log   D\right)\right),$ which does not depend on the compactification, \cite[2.4]{Fu14}. Further, we canonically have a quasi-albanese map $a_U:U\to T_U$, such that $a_U^*\left(T_1\left(T_U\right)\right)=T_1\left(U\right)$. $T_U$ is the quasi-albanese variety of $U$ which is a quasi-abelian variety and the quasi-albanese map $a_U$ is algebraic. We refer to \cite{Fu14}, \cite{Ii76} for the details of the quasi-albanese map.

To prove Theorem \ref{T:log pair} by applying the argument in the proof of the previous theorem, we need to construct such algebraic morphism of log-smooth pairs $f:\left(X,D\right)\to \left(P^{r,d},L\right)$. Ideally, we want to directly use the quasi-albanese map $a_U:U \to T^{r,d}$, and then compactify it and perform a log-resolution to get $f$. However, taking log-resolution may introduce new zero-loci for holomorphic log-one-forms. To keep track the zero-loci, we are only allowed to perform toroidal log-resolutions. That is the reason that we consider the following
\begin{lemma}\label{L:toroidal extension}
Fix a log smooth pair $\left(X,D\right)$, and denote $X\setminus D=U$. Assume that we have an algebraic morphism $f:U\to T^{r,d}$. Then there exists a toroidal log-resolution $\tau:\left(\hat{X},\hat{D}\right)\to \left(X,D\right)$ such that $f$ can be extended to get a morphism of log-pairs: $\hat{f}:\left(\hat{X},\hat{D}\right)\to \left(P^{r,d},L\right)$.
\end{lemma}
\begin{proof}
We first note that although we may not be able to extended $f$ onto $X$, the morphism $p\circ f:U\to A^d$ is well defined on $X$. This is because if not, by a sequence of blowing-ups of smooth center, we can get a morphism. However, each rational curve will map to a point on an abelian variety, hence these blow-ups are not needed.

Then we argue that it suffices to show the case that $r=1$. Actually, it is not hard to see that we only need to show the case that $\left(P^{r,d},L\right)$ is the $(\mathbb{P}^1)^r$-bundle compactification of $T^{r,d}$, Remark \ref{R:toric compactification}. Note that in this case, we can decompose
$$P^{r,d}=P^{1,d}_1\times_{A^d} P^{1,d}_2 \times_{A^d} ...\times_{A^d} P^{1,d}_r,$$
where $P^{1,d}_i$ is the $\mathbb{P}^1$-bundle compatification of a quasi-abelian variety $T^{1,d}_i$, with the indicated dimension. Hence to get a morphism of log-pairs $\hat{f}:\left(\hat{X},\hat{D}\right)\to \left(P^{r,d},L\right)$, we only need to get a morphism of log-pairs $\left(\hat{X},\hat{D}\right)\to \left(P^{1,d}_i,L\right)$ for each $i$. From now on, we only consider the case that $r=1$.

Denote the two components of $L$ by $L_1$ and $L_2$. It is evident that $L$ separates into two components by the construction of $P^{1,d}$. We also have that $L_1$ and $L_2$ are linearly equivalent (\cite[Exercise II.7.9]{Ha77}). Since $f$ is well-defined on $X$ except a co-dimension $2$ locus. Hence the zero order of $f^*L_1$ and $f^*L_2$ are well defined along each irreducible component $D_i$ of $D$. We denote by $D'$ and $D''$ the two sub-divisors of $D$ that consists of those irreducible components with positive zero order of $f^*L_1$ and $f^*L_2$ respectively. Since $L_1\cap L_2=\emptyset$, $D'$ and $D''$ have no common components. We say that a co-dimension-$2$ stratum $S$ of $D$ is \emph{bad} if $S$ is the generic locus of $D'_1\cap D''_1$, with $D'_1$ and $D''_1$ are irreducible components of $D'$ and $D''$ respectively. For the rest of the proof, we will show that \\
(1)If we have no bad stratum, then $f$ is well defined on $X$.\\
(2)We can find a sequence of toroidal resolutions on $(X,D)$ to eliminate bad strata.

For (1), if $f$ is not well defined on $X$, by a sequence of blowing-ups of smooth locus on $D$, we can get a morphism of log-pairs $\hat{f}:(\hat{X},\hat{D})\to (P^{1,d},L)$. Note that since we start with no bad stratum, it is evident that for each step of blow-up, we will not introduce new bad stratum. Let's consider the last step of the blow-ups and we denote it by $\pi:(\hat{X},\hat{D})\to (\overline{X},\overline{D})$, with $E\subset \hat{D}$ being the exceptional divisor. We only need to show that each $\mathbb{P}^1$ on $E$ maps to a point on $P^{1,d}$, which implies that this blow-up is not needed. Hence we can conclude (1) by induction. Otherwise, since each $\mathbb{P}^1$ on $E$ will map to a point on $A^d$, we can find a $\mathbb{P}^1$ on $E$ maps to a $\mathbb{P}^1$ fiber over a point $a\in A^d$ surjectively. In particular, there is one point $e_1\in E$ that maps to $L_1$ and another point $e_2\in E$ maps to $L_2$. Consider the two sub-divisors of $\hat{D}$: $\hat{D}':=\hat{f}^{-1}L_1$ and $\hat{D}'':=\hat{f}^{-1}L_2$. Obviously they have no common components, and $E$ does not belong to either one. However, we can find two irreducible components $\hat{D}_1$ and $\hat{D}_2$ such that $e_1\in \hat{D}_1\subset \hat{D}'$, and $e_2\in \hat{D}_2\subset \hat{D}''$. We denote by $\overline{D}_1$ and $\overline{D}_2$ the two irreducible components of $\overline{D}$ with their strict transform on $\hat{X}$ being $\hat{D}_1$ and $\hat{D}_2$ respectively. Let $\overline{S}$ be the stratum given by the generic locus of $\overline{D}_1\cap \overline{D}_2$ which is not empty by the construction. However it is not hard to see that $\overline{S}$ is a bad stratum. Hence we get a contradiction.

For (2), since $L_1$ and $L_2$ are linearly equivalent, we can find a rational function $y$ on $P^{1,d}$ such that its zero-locus and pole locus are $L_1$ and $L_2$ respectively. Since $f$ is well defined on $U$ and its image is in $T^{1,d}$, the zero-locus and the pole locus of the rational function $f^*y$ on $X$ has to be contained in $D$. Hence, around a stratum $S$, the generic locus of $\cap_{1\leq i\leq n}D_i$, up to a multiplication of a unit, we locally have $f^* y=x_1^{a_1}\cdot...\cdot x_n^{a_n}$, with $x_i$ are local functions that defines $D_i$, the irreducible components of $D$, and $a_i\in \mathbb{Z}$. Note that $a_i$ does not depend on the choice of the stratum. Pick a bad stratum $S$ of co-dimension-$2$. Assume $S$ is the generic locus of $D_1\cap D_2$. $S$ being bad is equivalent to that $a_1\cdot a_2<0$. We can find two co-prime integers $b_1, b_2$, such that $a_1\cdot b_2+a_2\cdot b_1=0$. By a blow-up of the ideal locally being $\left(x_1^{b_1},x_2^{b_2}\right)$, which is toroidal, we get $\pi:\left(X',D'+E\right)\to \left(X,D\right)$, where $D'$ is the strict transform of $D$ and $E$ is the exceptional divisor. 
Now by induction, we only need to show that those newly introduced co-dimension-$2$ strata of $D'+E$ are not bad. It is evident by noticing that the zero order of $(\pi\circ f)^*y$ along $E$ is $0$.

Note that $\left(X',D'+E\right)$ we construct above is not log-smooth in general, but after we eliminate all bad strata, we can perform a toroidal log-resolution to get a log-smooth pair we are after.
\end{proof}

Now, we are ready to show the proof of Theorem \ref{T:log pair} assuming Theorem \ref{T:over quasi-abelian variety}.

\begin{proof}[Proof of Theorem \ref{T:log pair}]
Following the argument of the proof of Theorem \ref{T:over P^{r,d}}, we are reduced to consider the rational map $g_\mathbf{a}:X\dashrightarrow P^{1,d}_\mathbf{a},$
which is well defined on $U$, $g_\mathbf{a}\big|_U:U\to T^{1,d}_\mathbf{a}$. Now apply the previous lemma, we can find a toroidal log-resolution $\left(\hat{X}, \hat{D}\right)\to (X,D)$ such that we honestly get a morphism of log pairs:
$$\hat{g}_\mathbf{a}:\left(\hat{X}, \hat{D}\right)\to \left(P^{1,d}_\mathbf{a},L\right).$$
Hence we can conclude the proof by the $r=0$ or $1$ case of Theorem \ref{T:over quasi-abelian variety} as in the proof of Theorem \ref{T:over P^{r,d}}.
\end{proof}

For the rest of the section, we state couple corollaries that follow by Theorem \ref{T:log pair}.

\begin{coro}
Given a smooth projective variety $X$, a global holomorphic one-form $\theta$ on $X$ with no zero-locus, and a smooth divisor $D$. If $\omega_X\left(D\right)$ is big, then $\theta|_D$ must have non-empty zero-locus as a global holomorphic one-form on $D$.
\end{coro}
\begin{proof}
Let $x_1,...,x_n$ be a local holomorphic coordinate system of $X$ and $x_1$ defines $D$. Then $\{dx_1/x_1, dx_2,...,dx_n\}$ gives a local basis of $\Omega^1_X\left(\log D\right)$. Since we know that it has no zero-locus as holomorphic one-form, according to Theorem \ref{T:log general}, $\theta$ has a zero-locus at some point $p\in D$ as a global section of $\Omega^1_X\left(\log D\right)$. Hence it locally looks like 
$\theta=x_1 g_1dx_1/x_1+g_2dx_2...+g_ndx_n$, where $g_2\left(p\right)=...=g_n\left(p\right)=0$. Hence $\theta|_D\in H^0\left(D,\Omega_D\right)$ has zero-locus at $p$.
\end{proof}

\begin{definition}
Given a projective morphism of log pairs $f:\left(X,D^X\right)\to \left(Y,D^Y\right)$ and assume that $\left(X,D^X_{\text{red}}\right)$ and $\left(Y,D^Y_{\text{red}}\right)$ are log-smooth. We say that $f$ is log-smooth if the cokernel sheaf of the natural map 
$$f^*\Omega^1_Y\left(\log   D^Y\right)\to \Omega^1_X\left(\log   D^X\right),$$
is locally-free.
\end{definition}

\begin{coro}\label{C:log smooth}
If $f:\left(X,D\right)\to \left(P^{r,d},L\right)$ is a log-smooth morphism, from a log-smooth pair $\left(X,D\right)$ onto the canonical  $\mathbb{P}^r$-bundle compactification of a quasi-abelian variety $T^{r,d}$, of dimension $m=r+d$, then $m\leq \dim X- \kappa\left(X,D\right)$.
\end{coro}
\begin{proof}
It follows immediately by Theorem \ref{T:log pair}. See also Remark \ref{R:toric compactification}.
\end{proof}

\begin{coro}\label{C:log smooth*}
Let $\left(X,D\right)$ be a projective log-smooth pair of log general type. Assume that we have a surjective projective morphism $f: X\setminus D \to T^{r,d}$, where $T^{r,d}$ is a quasi-abelian variety.  Then, $f$ is not smooth.
\end{coro}
\begin{proof}
We can find a log-resolution $\left(\hat{X},\hat{D}\right)\to \left(X,D\right)$ such that it induces a morphism of log-pairs 
$$\hat{f}:\left(\hat{X},\hat{D}\right)\to \left(P^{r,d},L\right),$$
with $\hat{f}^{-1}L= \hat{D}$.
Now by Theorem \ref{T:over P^{r,d}}, we have $Z\left(\hat{f}^*\omega\right)\neq \emptyset$, for any $\omega\in H^0\left(\Omega^1_{P^{r,d}}\left(\log L\right)\right)$. However, by Lemma \ref{L:g}, (see also Step 0 of the proof of Theorem \ref{T:over quasi-abelian variety} in Section 4,) for general such $\omega$, $Z\left(\hat{f}^*\omega\right)\cap (\hat{X}\setminus \hat{D})\neq \emptyset.$ We refer \cite{W16} for more details, and a different proof based on the structure of Higgs bundles.
\end{proof}
We use the following lemma in the proof above, \cite[Lemma 8]{W16}. 
\begin{lemma}\label{L:g}
Given a morphism $f: X\to P^{r,d}$, where $X$ is smooth and quasi-projective and $P^{r,d}$ is the $\mathbb{P}^r$-bundle compactification of a quasi-abelian variety $T^{r,d}$, with $L$ being the boundary divisor on $P^{r,d}$. Let $D=\left(f^*L\right)_{red}$ and assume that $\left(X,D\right)$ is log-smooth. Then for general $\theta \in H^0\left(\Omega^1_{P^{r,d}}\left(\log   L\right)\right)$, $f^* \theta\in H^0 \left(\Omega^1_X\left(\log  D\right)\right)$ has a simple log pole on every point of $D$. In particular, it does not vanish at any point of $D$.
\end{lemma}

Another application of Theorem \ref{T:log pair} is that we can give an answer to the algebraic hyperbolicity part of Shafarevich's conjecture, with general fibers being of log-general type, which is the special case of the following theorem by taking $T^{r,d}=\mathbb{G}^1_m$ or $A^d$. We refer to \cite[\S 16]{HK10} for details of this topic. Note that a projective morphism of log-smooth pairs $f:(U,D^U)\to \mathbb{G}^1_m$ being log-smooth is equivalent to the existence of $(X,D^X)$, a projective log-smooth compactification of $(U,D^U)$ such that $g:(X,D^X)\to (\mathbb{P}^1,L)$, the induced morphism of log-smooth pairs by $f$, is log-smooth. It can be very complicated in the higher dimensional cases. A similar question has been asked and we refer to \cite{AKMW} for the details on this topic.

\begin{definition}
Given a flat dominant projective morphism of log-smooth pairs $f:(X,D^X)\to (Y,D^Y)$, with connected fibers and the generic fiber $(X_\eta, D^X_\eta)$ being Kawamata-log-terminal (klt) and of log-general type. We say that it is \emph{birationally isotrivial} if the two log-fibers $(X_a, D^X_a)$ and $(X_b, D^X_b)$ have the same log-canonical model (\cite{BCHM}), for any two general closed points $a$ and $b$ on $Y$.
\end{definition}

\begin{coro}\label{C:fiber of lg}
Given a log-smooth surjective morphism of log-pairs $f:\left(X, D\right)\to \left(P^{r,d},L\right)$, with $\left(X, D_{red}\right)$ being log-smooth and the generic fiber $(X_\eta, D_\eta)$ being klt and of log-general type, then $f$ is birationally isotrivial.
\end{coro}
\begin{proof}
We only need to show $\text{Var}(f)=0$, in the sense of \cite[Definition 9.3]{KP16}. 
By Corollary \ref{C:log smooth}, we have  
$$\kappa\left(X,D\right)\leq \dim X_\eta.$$
By Theorem \ref{T:subadditivity}, we have
$$\kappa\left(X,D\right)\geq \kappa\left(X_\eta, D_\eta\right)+\text{Var}(f).$$
Since $\dim X_\eta=\kappa\left(X_\eta, D_\eta\right)$ by the assumption of the theorem, we get that $\text{Var} (f)=0.$
\end{proof}
\begin{remark}
If we only assume that the generic fiber $(X_\eta, D_\eta)$ is log-canonical and of log-general type, since $\left(X_\eta, D_\eta\right)$ being of log-general type is an open condition on the coefficients of the boundary divisor, we can consider
$$f':\left(X, \left(1-\epsilon\right) D\right)\to  \left(P^{r,d},L\right),$$
for $0<\epsilon \ll 1$. Then the generic fiber of $f'$ is klt and of log-general type. Following the same arguement above, we can show that $f'$ is birationally isotrivial.
\end{remark}
In the proof of the previous theorem, we used the following theorem, which can be easily deduced from  \cite[Theorem 9.5, Theorem 9.6]{KP16}.
\begin{theorem}\label{T:subadditivity}
Let $ f: \left(X,D\right) \rightarrow \left(Y,E\right) $ be a surjective morphism of projective log canonical pairs with both $\left(X,D_{\text{red}}\right)$ and $\left(Y,E\right)$ are log-smooth. Assume that $\lfloor D\rfloor$ contains $f^{-1}E_{\text{red}}$ and the generic fiber $\left(X_{\eta},D_{\eta}\right)$ is of log-general type. Then 
  \begin{equation*}
    \kappa\left(X,D\right)\geq\kappa\left(X_{\eta},D_{\eta}\right) +\kappa\left(Y, E\right).
  \end{equation*}
    Further, if $\left(X_\eta, D_\eta\right)$ is klt, then
  \begin{equation*}
   \kappa\left(X,D\right)\geq \kappa \left(X_{\eta},D_{\eta}\right)+ max \{\kappa\left(Y, E\right), \text{Var}(f)\}.
  \end{equation*}
\end{theorem}

\section{A Simplified proof of the non-log version of Theorem \ref{T:over quasi-abelian variety}}
In this section, we show a simplified proof of the the following theorem, which is the main result of \cite{PS14}, without using the generic vanishing of Hodge modules on abelian varieties, which they introduced in \cite{PS13}. The method is still mainly based on \cite{PS13}.
\begin{theorem}\label{T:over abelian variety}
Given a morphism $f:X\to A$, if there is a positive integer $k$ and an ample line bundle $\mathcal{A}$ on $A$, such that 
$$H^0\left(X, \omega_X^{\otimes k}\otimes f^*\mathcal{A}^{-1}\right)\neq 0,$$
then $Z\left(f^*\omega\right)\neq \emptyset$, for any $\omega\in H^0\left(A, \Omega^1_A\right).$
\end{theorem}
\begin{proof}
\emph{Step 1.}
We have an étale morphism $[k]: A\to A$ which is defined by multiplying by $k$. Apply the finite base change and let $f':X'\to A$ be the fiber product. Since $[k]$ is étale, we have that $X'$ is smooth and it suffices to prove the theorem for $f': X'\to A$.

\begin{center}
\begin{tikzcd}
X'\arrow{r}{\phi}\arrow{d}{f'}
&X\arrow{d}{f}
\\
A\arrow{r}{[k]}
&A
\end{tikzcd}
\end{center}

We have an inclusion $\mathcal{A}^{\otimes k}\to [k]^*\mathcal{A}$ \cite[II, 6]{Mu08}. Hence

\begin{align*}
&\phi^*\left(\omega_X^{\otimes k}\otimes f^*\left(\mathcal{A}^{-1}\right)\right)\\
=& \phi^*\omega_X^{\otimes k}\otimes f'^*[k]^* \mathcal{A}^{-1}\\
\subset &\left(\omega_{X'}\otimes f'^*\mathcal{A}^{-1}\right)^{\otimes k}
\end{align*}
which, by assumption, has a global section. 

Hence we reduce to the case that there exists a positive integer $k$ such that 
$$H^0\left(X, \left(\omega_{X}\otimes f^*\mathcal{A}^{-1}\right)^{\otimes k}\right)\neq 0.$$

\emph{Step 2.}
Let $\mathcal{B}=\omega_{X}\otimes f^*\mathcal{A}^{-1}$, a line bundle over $X$. Do the cyclic cover induced by a section $s$ of $\mathcal{B}^{\otimes k}$ and resolve it. We get $\psi: Y\to X.$ Hence we have the following commutative diagram
\begin{center}
\begin{tikzcd}
Y \arrow{r}{\psi}\arrow{rd}{g}
&X \arrow{d}{f}
\\

&A.
\end{tikzcd}
\end{center}

By the construction, there is a tautological section of $\psi^* \mathcal{B}$ on $Y$. Hence we have a natural injection, 
$$\psi^*\mathcal{B}^{-1}\to\mathscr{O}_Y,$$
which induces the following injection
\begin{equation}\label{e:inj(a)}
\psi^*\left(\mathcal{B}^{-1}\otimes \Omega_X^k \right) \to \Omega_Y^k.
\end{equation}
Actually, both of them are isomorpisms over the complement of the zero-locus of $s$.

\emph{Step 3.}
Let $d=\dim A.$
We have that $\Omega_{A}^1$ is a trivial $d$-dimensional vector bundle over $A$. Denote $V= H^0\left(A,\Omega_{A}^1\right),$ and denote the space of cotangent of $A$ by $T^*_A$. Hence we have
$$T^*_A= A\times V.$$
Consider the following commutative diagram, which contains all the morphisms we will need in the proof:
\begin{center}
\begin{tikzcd}
Y\arrow[leftarrow]{r}{p_Y}\arrow[bend right]{dd}[swap]{g}\arrow{d}{\psi}&
Y\times V\arrow[bend left=45]{dd}{\hat{g}}\arrow{d}{\hat{\psi}}
\\
X\arrow[leftarrow]{r}{p_{X}}\arrow{d}{f}&
X\times V\arrow{d}{\hat{f}}

\\
A\arrow[leftarrow]{r}{p_{A}}&
A\times V\arrow{r}{p_V}&
V,
\end{tikzcd}
\end{center}
where all the morphisms are natural projections or naturally induced by $f$ and $\psi$.

Let $n=\dim X=\dim Y$. Denote
  \begin{equation*} 
  C_{Y}=[p_Y^*\mathscr{O}_Y \to p_Y^*\Omega_Y^1 \to ... \to p_Y^*\Omega_Y^n],
  \end{equation*}
placed in cohomological degrees $-n, -n+1, ... , 0$, which is the Koszul complex given by the pullback of the tautological section of $p_{A}^*\Omega_{A}^1$. Note that $p_{Y}$ is an affine morphism, and $p_{Y*}\left(C_{Y}\right)$ is the following graded complex
$$C_{Y, \bullet}=[\mathscr{O}_Y\otimes S_{\bullet-d} \to \Omega_Y^1\otimes S_{\bullet-d+1} \to 
... \to \Omega_Y^n\otimes S_{\bullet-d+n}],
$$
where $S_\bullet:= \text{Sym} V^*$. The differential in the complex is induced by the evaluation morphism $V\otimes \mathscr{O}_Y \to \Omega_Y^1.$
We define $C_{X}$, $C_{X,\bullet}$ in a similar way.
Denote $\mathcal{L}=g^*\mathcal{A}$ and $\hat{\mathcal{L}}= p_Y^*\left(\mathcal{L}\right).$

\begin{claim}\label{C:claim (a)}
$$p_{V*}R^0\hat{g}_*\left(\hat{\mathcal{L}}^{-1}\otimes C_{Y}\right)$$
is torsion free.
\end{claim}

\emph{Step 4.}
We continue the proof assuming the claim.

We set $\mathcal{F}$ as the image of the map 
$$R^0\hat{g}_*\left(\hat{\mathcal{L}}^{-1}\otimes\hat{\psi}^*\left(p_X^*\mathcal{B}^{-1}\otimes C_{X}\right)\right)\to R^0\hat{g}_*\left(\hat{\mathcal{L}}^{-1}\otimes C_{Y}\right)$$
induced by the injection (\ref{e:inj(a)}).

Pick any global one-form 
$$\theta\in H^0\left(A, \Omega^1_{A}\right)=V.$$
The restriction of $p_X^*\mathcal{B}^{-1}\otimes C_{X}$ on the fiber of $p_V\circ \hat{f}: X\times V\to V$ over $\theta$ is just
$$\mathcal{B}^{-1}\to\mathcal{B}^{-1}\otimes\Omega^1_X\to...\to\mathcal{B}^{-1}\otimes\Omega^n_X,$$
the Koszul complex defined by $f^*\left(\theta\right)$ and twisted by $\mathcal{B}^{-1}$. Since the zero-locus of $f^*\left(\theta\right)$ is empty if and only if the above complex is exact, to prove the theorem, it suffices to prove $p_V\left(\text{supp}\mathcal{F}\right)=V$.

Note that $p_{A}$ and $p_Y$ are affine, so we have that $p_{A*}$ and $p_{Y*}$ are exact, and $p_{A*}\circ R^0\hat{g}_*= R^0 g_*\circ p_{Y*}.$ Hence, $p_{A*}\mathcal{F}$ is a graded $p_{A*}\mathscr{O}_{A\times V}$-module $\mathcal{F}_\bullet$ given by the image of
$$p_{A*}R^0\hat{g}_*\left(\hat{\mathcal{L}}^{-1}\otimes\hat{\psi}^*\left(p_X^*\mathcal{B}^{-1}\otimes C_{X}\right)\right)\to p_{A*}R^0\hat{g}_*\left(\hat{\mathcal{L}}^{-1}\otimes C_{Y}\right),$$
which is the same as the image of
$$
R^0g_*\left(\mathcal{L}^{-1} \otimes\psi^*\left(\mathcal{B}^{-1}\otimes C_{X,\bullet}\right)\right)
\to R^0g_*\left(\mathcal{L}^{-1} \otimes C_{Y,\bullet}\right).
$$

If $p_V\left(\text{Supp}\mathcal{F}\right)\neq V$, the subsheaf 
  \begin{equation*}
  p_{V*}\mathcal{F}\subset p_{V*} R^0\hat{g}_*\left(\hat{\mathcal{L}}^{-1}\otimes C_{Y}\right)
  \end{equation*}
would then be torsion and hence zero. Therefore, since $V$ is a vector space, $H^0\left(A, p_{A*}\mathcal{F}\right)=H^0\left(A,\mathcal{F}_\bullet\right)=0$. Recall that $\mathcal{B}=\omega_{X}\otimes f^* \mathcal{A}^{-1}$ and $\mathcal{L}=g^*\mathcal{A}.$ We have 
\begin{align*}
\mathcal{F}_{-n+r}= g_*\left(\mathcal{L}^{-1} \otimes \psi^*\left(\mathcal{B}^{-1}\otimes \omega_X\right)\right)= g_* \mathscr{O}_Y
\end{align*}
This forces $H^0\left(g_* \mathscr{O}_Y \right)=0,$
which is absurd.

\emph{Step 5.}
Finally, we show the proof of Claim \ref{C:claim (a)}.

Denote $\mathcal{E}=\mathbf{R}p_{V*}R^0\hat{g}_*\left(\hat{\mathcal{L}}\otimes C_{Y}\right).$ 
We first show that for $l>0$,
\begin{equation}\label{e:ve(a)}
\mathcal{H}^l \mathcal{E}=0.
\end{equation}
In particular, $\mathcal{E}$ is a sheaf.

Since both $p_A$ and $p_Y$ are affine, we have 
\begin{align*}
&H^l\left(A\times V, R^0\hat{g}_*\left(\hat{\mathcal{L}}\otimes C_{Y}\right)\right) \\
\simeq&H^l\left(A, p_{A*} R^0\hat{g}_*\left(\hat{\mathcal{L}}\otimes C_{Y}\right)\right)\\
\simeq&H^l\left(A,  R^0g_* p_{Y*}\left(\hat{\mathcal{L}}\otimes C_{Y}\right)\right) \\
=&H^l\left(A,  R^0g_* \left(C_{Y,\bullet}\right) \otimes \mathcal{A}\right)\\
=&0
\end{align*}
The last vanishing is due to Laumon's formula (\cite[Lemma 15.1]{PS14}, or Proposition \ref{P:gr direct}) and \cite[Lemma 2.5]{PS14} or Proposition \ref{P:v}. More precisely, we have 
$$R^0g_* \left(C_{Y,\bullet}\right)\simeq \text{Gr}^F_{\bullet}\mathcal{R}^0g_+\widetilde{\omega}_Y,$$
and 
$$H^l\left(\text{Gr}^F_{\bullet}\mathcal{R}^0g_+\widetilde{\omega}_Y\otimes \mathcal{A}\right)=0,$$
for $l>0$.

Hence, due to the degeneration of the Leray spectral sequence induced by $p_{V*}$, we have 
\begin{align*}
&H^0\left(V, \mathcal{H}^l \mathcal{E}\right)\\
=&H^0\left(V, R^l p_{V*} R^0\hat{g}_*\left(\hat{\mathcal{L}}\otimes C_{Y}\right)\right)\\
\simeq&H^l\left(P\times V, R^0\hat{g}_*\left(\hat{\mathcal{L}}\otimes C_{Y}\right)\right)\\
=&0,
\end{align*} 
and so that (\ref{e:ve(a)}) follows.

$p_{V*}R^0\hat{g}_*\left(\hat{\mathcal{L}}^{-1}\otimes C_{Y}\right)$ being torsion free is due to the following relation:
\begin{equation}\label{e:claim dual(a)}
p_{V*}R^0\hat{g}_*\left(\hat{\mathcal{L}}^{-1}\otimes C_{Y}\right)=(-1_V)^*\mathbf{R}^0\mathcal{H}om\left(\mathcal{E}, \mathscr{O}_V\right),
\end{equation}
where $\left(-1_V\right)$ is the involution on $V$ by multiplying $-1$. To show (\ref{e:claim dual(a)}), by Grothendieck duality, we have that 
\begin{align*}
&\mathbf{R}\mathcal{H}om\left(\mathcal{E},\mathscr{O}_V\right)\nonumber\\
=&\mathbb{D}_V\left(\mathbf{R}p_{V*} R^0\hat{g}_*\left(\hat{\mathcal{L}}\otimes C_{Y}\right)\right)[-d]\nonumber\\
\simeq&\mathbf{R}p_{V*} \mathbb{D}_{A\times V}\left( R^0\hat{g}_*\left(\hat{\mathcal{L}}\otimes C_{Y}\right)\right)[-d].\\
=&\mathbf{R}p_{V*} \mathbf{R}\mathcal{H}om\left( R^0\hat{g}_*\left(\hat{\mathcal{L}}\otimes C_{Y}\right), \mathscr{O}_{A\times V}[d]\right).
\end{align*}
By \cite[Proposition 2.11]{PS13}, (or Proposition \ref{P:direct} by taking $D^Y=0$ and $r=0$,) we have 
$$R^0\hat{g}_*\left(C_{Y}\right)\simeq \mathcal{G}\left(\mathcal{H}^0g_+\widetilde{\omega}_X\right),$$
the corresponding coherent sheaf on $T^*_A=A\times V$ of the mixed Hodge module $\mathcal{H}^0g_+\widetilde{\omega}_X$ on $A$. By the formula of taking duality in mixed Hodge modules \cite[Theorem 2.3]{PS13}, (or Proposition \ref{P:dual gr} by taking $H=0$,) and considering that the dual Hodge module of $\mathcal{R}^0g_+\widetilde{\omega}_X$ is itself up to a Tate twist, we obtain
$$\mathbf{R}\mathcal{H}om\left(\mathcal{E},\mathscr{O}_V\right)=\left(-1_V\right)^*\mathbf{R}p_{V*}R^0\hat{g}_*\left(\hat{\mathcal{L}}^{-1}\otimes C_{Y}\right),$$
which implies (\ref{e:claim dual(a)}).
\end{proof}

\section{Proof of Theorem \ref{T:over quasi-abelian variety}}
The proof is inspired by M. Popa and C. Schnell's work \cite{PS14}. 
The idea of the proof is similar to the proof of Theorem \ref{T:over abelian variety}.
\begin{proof}
\emph{Step 0.}
Let's consider the set 
$$S=\{\omega \in H^0\left(\Omega^1_{P^{r,d}}\left(\log   L\right)\right)\mid Z\left(f^*\omega\right)\neq \emptyset\}.$$
It is closed as a subset of $H^0\left(\Omega^1_{P^{r,d}}\left(\log   L\right)\right)$. Considering Lemma \ref{L:g}, it is easy to see that the two statements in the theorem are equivalent to each other. Hence, we can focus on the part over $T^{r,d}$. In other words, we only need to prove the theorem for any log-smooth model $\left(X',D'\right)$ over $\left(X,D\right)$, that is identical over $T^{r,d}$. In the rest of the proof, we use $P$ and $T$ to replace $P^{r,d}$ and $T^{r,d}$ respectively to simplify the notations. 

Actually, we can also restrict ourselves to the case that $r=0$ or $1$, which suffices to show the statement in general. See the second half of the proof of Theorem \ref{T:over P^{r,d}} in Section 2 for details.

\emph{Step 1.}
We have a finite morphism $[n]: P\to P$ which is defined in the following way. In the interior part $T$, it is defined by multiplying by $n$ (if we use addition for the group structure on $T$). It can be canonically extended onto the boundary $L$ \cite{W16}. Note that it is only ramified over the boundary $L$ of degree $n$. Take $n=k$, and apply the finite base change, and let $\left(X_n,D_n\right)$ be the nomalization of the fiber product. $f_n:X_n\to P$ is the induced morphism. Let $\left(X', D'\right)$ be a log-resolution of $\left(X_n, D_n\right)$, which can be achieved by only blowing up loci contained in $f_n^{-1}\left(L\right)$. According to Step 0, it suffices to prove the theorem for $f': X'\to P$.

\begin{center}
\begin{tikzcd}
X'\arrow{r}{\pi}\arrow{rrd}[swap]{f'}\arrow[bend left]{rrr}{\phi}
&X_n\arrow{r}\arrow{rd}{f_n}
&X\times _{P}P \arrow{d}\arrow{r}
&X\arrow{d}{f}
\\
&
&
P\arrow{r}{[k]}
&P
\end{tikzcd}
\end{center}

Note that we have $\phi^*\omega_X\left(D\right) \subset \omega_{X'}\left(D'\right)$, and we have an inclusion $p^*\mathcal{A}^{\otimes k}\to [k]^* p^*\mathcal{A}$ \cite[II, 6]{Mu08}. Hence
\begin{align*}
&\phi^*\left(\left(\omega_X\left(D\right)\right)^{\otimes k}\otimes f^*\left(p^*\mathcal{A}^{-1}\otimes \mathscr{O}_{P^{r,d}}\left(-L\right)\right)\right)\\
=& \left(\phi^*\omega_X\left(D\right)\right)^{\otimes k} \otimes f'^* [k]^*\mathscr{O}_{P}\left(-L\right)\otimes  f'^*[k]^* p^*\mathcal{A}^{-1}\\
\subset &\left(\omega_{X'}\left(D'\right)\right)^{\otimes k} \otimes f'^* \mathscr{O}_{P}\left(-kL\right)\otimes f'^* p^* \mathcal{A}^{-k}\\
=&\left(\omega_{X'}\left(D'\right) \otimes f'^*\mathscr{O}_{P}\left(-L\right)\otimes f'^*p^* \mathcal{A}^{-1}\right)^{\otimes k},
\end{align*}
which, by assumption, has a global section. Denote $E'=D'-f'^{-1}L.$ Note that $\mathscr{O}_{X'}\left(D'-E'\right)\subset f'^* \mathscr{O}_{P}\left(L\right)$, we have that 
$$H^0\left(X', \left(\omega_{X'}\left(E'\right)\otimes f'^*p^* \mathcal{A}^{-1}\right)^{\otimes k}\right)\neq 0.$$
Hence we reduce to the case that there exists a positive integer $k$ such that 
$$H^0\left(X, \left(\omega_{X}\left(E\right)\otimes f^*p^* \mathcal{A}^{-1}\right)^{\otimes k}\right)\neq 0,$$
where $E=D-f^{-1}L$.

\emph{Step 2.}
Let $\mathcal{B}=\omega_{X}\left(E\right)\otimes f^*p^* \mathcal{A}^{-1}$, a line bundle over $X$. Do the cyclic cover induced by a section $s$ of $\mathcal{B}^{\otimes k}$ and resolve it. We get $\psi: Y\to X.$ Denote $D^Y=\psi^{-1}D$ which can also be assumed to have normal crossings after further blowing up if necessary. We denote by $g:\left(Y,D^Y\right)\to \left(P,L\right)$  the induced morphism. Hence we have the following commutative diagram
\begin{center}
\begin{tikzcd}
\left(Y,D^Y\right)\arrow{r}{\psi}\arrow{rd}{g}
&\left(X,D\right)\arrow{d}{f}
\\

&\left(P,L\right).
\end{tikzcd}
\end{center}

By the construction, there is a tautological section of $\psi^* \mathcal{B}$ on $Y$. Hence we have a natural injection, 
$$\psi^*\mathcal{B}^{-1}\to\mathscr{O}_Y,$$
which induces the following injection
\begin{equation}\label{e:inj}
\psi^*\left(\mathcal{B}^{-1}\otimes \Omega_X^k\left(\log   \:D\right)\right)\to \Omega_Y^k\left(\log   \:D^Y\right).
\end{equation}
Actually, both of them are isomorpisms over the complement of the zero-locus of $s$.

\emph{Step 3.}
Let $m=r+d=\dim P.$
We have that $\Omega_{P}^1\left(\log    L\right)$ is a trivial $m$-dimensional vector bundle over $P$. Denote $V=H^0\left(P,\Omega_{P}^1\left(\log L\right)\right),$ and denote the space of log-cotangent of $(P,L)$ by $T^*_{(P,L)}$. Hence we have
$$T^*_{(P,L)}= P\times V.$$
Consider the following commutative diagram, which contains all the morphisms we will need in the proof:
\begin{center}
\begin{tikzcd}
Y\arrow[leftarrow]{r}{p_Y}\arrow[bend right]{dd}[swap]{g}\arrow{d}{\psi}&
Y\times V\arrow[bend left=45]{dd}{\hat{g}}\arrow{d}{\hat{\psi}}
\\
X\arrow[leftarrow]{r}{p_{X}}\arrow{d}{f}&
X\times V\arrow{d}{\hat{f}}

\\
P\arrow[leftarrow]{r}{p_{P}}\arrow{d}{p}&
P\times V\arrow{r}{p_V}\arrow{d}{\hat{p}}&
V\\
A^d\arrow[leftarrow]{r}{p_A}&
A^d\times V
\end{tikzcd}
\end{center}
where all the morphisms are natural projections or naturally induced by $f$ and $\psi$.

Let $n=\dim X=\dim Y$. Denote
  \begin{equation*} 
  C_{Y, D^Y}=[p_Y^*\mathscr{O}_Y \to p_Y^*\Omega_Y^1\left(\log   D^Y\right) \to ... \to p_Y^*\Omega_Y^n\left(\log    D^Y\right)],
  \end{equation*}
placed in cohomological degrees $-n, -n+1, ... , 0$, which is the Koszul complex given by the tautological section of $p_{P}^*\Omega_{P}^1\left(\log   L\right)$. Note that $p_{Y}$ is an affine morphism, and $p_{Y*}\left(C_{Y,D^Y}\right)$ is the following graded complex
$$C_{Y, D^Y, \bullet}=[\mathscr{O}_Y\otimes S_{\bullet-m} \to \Omega_Y^1\left(\log    D^Y\right)\otimes S_{\bullet-m+1} \to 
... \to \Omega_Y^n\left(\log    D^Y\right)\otimes S_{\bullet-m+n}],
$$
where $S_\bullet:= \text{Sym} V^*$. The differential in the complex is induced by the evaluation morphism $V\otimes \mathscr{O}_Y \to \Omega_Y^1.$
We define $C_{X,D}$, $C_{X,D,\bullet}$ in a similar way.
Denote 
$$\mathcal{L}=g^*p^*\mathcal{A}\otimes \mathscr{O}_Y\left(g^{-1}L\right),$$
and
$$\hat{\mathcal{L}}= p_Y^*\left(\mathcal{L}\right).$$
\begin{claim}[Main Claim]\label{C:main claim}
If $r=0$ or $1$,
$$p_{V*}R^0\hat{g}_*\left(\hat{\mathcal{L}}^{-1}\otimes C_{Y,D^Y}\right)$$
is torsion free.
\end{claim}
\emph{Step 4.}
We continue the proof assuming the claim.

We set $\mathcal{F}$ as the image of the map 
$$R^0\hat{g}_*\left(\hat{\mathcal{L}}^{-1}\otimes\hat{\psi}^*\left(p_X^*\mathcal{B}^{-1}\otimes C_{X,D}\right)\right)\to R^0\hat{g}_*\left(\hat{\mathcal{L}}^{-1}\otimes C_{Y,D^Y}\right)$$
induced by the injection (\ref{e:inj}).

Pick any global log-one-form 
$$\theta\in H^0\left(P, \Omega^1_{P}\left(\log   L\right)\right)=V.$$
The restriction of $p_X^*\mathcal{B}^{-1}\otimes C_{X,D}$ on the fiber of $p_V\circ \hat{f}: X\times V\to V$ over $\theta$ is just
$$\mathcal{B}^{-1}\to\mathcal{B}^{-1}\otimes\Omega^1_X\left(\log   \:D\right)\to...\to\mathcal{B}^{-1}\otimes\Omega^n_X\left(\log   \:D\right),$$
the Koszul complex defined by $f^*\left(\theta\right)$ and twisted by $\mathcal{B}^{-1}$. Since the zero-locus of $f^*\left(\theta\right)$ is empty if and only if the above complex is exact, to prove the theorem, it suffices to prove $p_V\left(\text{Supp}\mathcal{F}\right)=V$.

Note that $p_{P}$ and $p_Y$ are affine, so we have that $p_{P*}$ and $p_{Y*}$ are exact, and $p_{P*}\circ R^0\hat{g}_*= R^0 g_*\circ p_{Y*}.$ Hence, $p_{P*}\mathcal{F}$ is a graded $p_{P*}\mathscr{O}_{P\times V}$-module $\mathcal{F}_\bullet$ given by the image of
$$p_{P*}R^0\hat{g}_*\left(\hat{\mathcal{L}}^{-1}\otimes\hat{\psi}^*\left(p_X^*\mathcal{B}^{-1}\otimes C_{X,D}\right)\right)\to p_{P*}R^0\hat{g}_*\left(\hat{\mathcal{L}}^{-1}\otimes C_{Y,D^Y}\right),$$
which is the same as the image of
$$
R^0g_*\left(\mathcal{L}^{-1} \otimes\psi^*\left(\mathcal{B}^{-1}\otimes C_{X,D,\bullet}\right)\right)
\to R^0g_*\left(\mathcal{L}^{-1} \otimes C_{Y,D^Y,\bullet}\right).
$$

If $p_V\left(\text{Supp}\mathcal{F}\right)\neq V$, the subsheaf 
  \begin{equation*}
  p_{V*}\mathcal{F}\subset p_{V*} R^0\hat{g}_*\left(\hat{\mathcal{L}}^{-1}\otimes C_{Y,D^Y}\right)
  \end{equation*}
would then be torsion and hence zero. Therefore, since $V$ is a vector space, $H^0\left(P, p_{P*}\mathcal{F}\right)=H^0\left(P,\mathcal{F}_\bullet\right)=0$. Recall that $\mathcal{B}=\omega_{X}\left(E\right)\otimes f^* p^*\mathcal{A}^{-1}$ and $\mathcal{L}=g^*p^*\mathcal{A}\otimes \mathscr{O}_Y\left(g^{-1}L\right).$ We have 
\begin{align*}
\mathcal{F}_{-n+r}=& g_*\left(\mathcal{L}^{-1} \otimes \psi^*\left(\mathcal{B}^{-1}\otimes \omega_X\left(D\right)\right)\right)\\
=& g_*\left(\mathscr{O}_Y\left(\psi^*\left(D-E\right)-g^{-1}L\right)\right)\\
=&g_*\left(\mathscr{O}_Y\left(\psi^*f^{-1}L-g^{-1}L\right)\right)
\end{align*}
This forces 
$$H^0\left(P, g_*\left(\mathscr{O}_Y\left(\psi^*f^{-1}L-g^{-1}L\right)\right)\right)=0,$$
which is absurd, since $\psi^*f^{-1}L-g^{-1}L$ is effective over $Y$.
\end{proof}

\section{Log comparison of mixed Hodge modules and related vanishings}
In this section, we first recall some results about the logarithmic comparison proved in \cite{W17a}. Then we deduce some vanishing theorems that will be used in proving the Main Claim 

Notations as in \cite{W17a}, we use strict right $\widetilde{\mathscr{D}}$-modules to represent mixed Hodge modules, forgetting the weight filtration.

Given a mixed Hodge module $\mathcal{M}$ on a smooth variety $X$, and given a normal crossing boundary divisor $D=D_1+...+D_m$ on $X$, we introduce two strict $\widetilde{\mathscr{D}}_{\left(X,D\right)}$-modules:
$$\mathcal{M}_{*D}=\mathbf{V}^D_\mathbf{0}\mathcal{M}[*D]:=\cap_i V^{D_i}_0\mathcal{M}[*D],$$
$$\mathcal{M}_{!D}=\mathbf{V}^D_{<\mathbf{0}}\mathcal{M}[!D]:=\cap_i V^{D_i}_{<0}\mathcal{M}[!D],$$
where $\mathcal{M}[*D]$ (resp. $\mathcal{M}[!D]$) is the localization (resp. dual localization) of the mixed Hodge module $\mathcal{M}$ along $D$ \cite[9]{SS16}, and $V^{D_i}$ is the $V$-filtration respect to $D_i$. Note that $V^{D_i}_{<0}$ only depends on $\mathcal{M}\big|_{X\setminus D_i}$. In particular, we have
$$\mathbf{V}^D_{<\mathbf{0}}\mathcal{M}[!D]=\mathbf{V}^D_{<\mathbf{0}}\mathcal{M}=\mathbf{V}^D_{<\mathbf{0}}\mathcal{M}[*D].$$
\begin{prop}[\cite{W17a}]\label{P:log rep}
Notations as above, assume that $\mathcal{M}$ is of normal crossing type respect to $D$ or $D$ is smooth. Then we have
$$\mathcal{M}_{*D}\otimes^{\mathbf{L}}_{\widetilde{\mathscr{D}}_{\left(X,D\right)}}\widetilde{\mathscr{D}}_{X}\simeq \mathcal{M}[*D],$$
 and  
 $$\mathcal{M}_{!D}\otimes^{\mathbf{L}}_{\widetilde{\mathscr{D}}_{\left(X,D\right)}}\widetilde{\mathscr{D}}_{X}\simeq\mathcal{M}[!D].$$ 
In particular, we have the following quasi-isomorphisms in the derived category of graded $\widetilde{\mathbb{C}}$-modules:
\begin{align*}
\textsl{Sp}_{\left(X,D\right)}\mathcal{M}_{*D}\simeq& \textsl{Sp}_X\mathcal{M}[*D],\\
\textsl{Sp}_{\left(X,D\right)}\mathcal{M}_{!D}\simeq& \textsl{Sp}_X\mathcal{M}[!D].
\end{align*}
\end{prop}
Let $f:\left(X,D^X\right)\to \left(Y,D^Y\right)$ be a projective morphism between two log-smooth pairs. Assume that $D^Y$ is smooth, which means $D^Y$ has only one component or all components of $D^Y$ do not intersect each other.
\begin{prop}[\cite{W17a}]
Notations as above and assume $D^X=f^{-1}D^Y$. Given a mixed Hodge module $\mathcal{M}$ of normal crossing type respect to a normal crossing divisor $D'$ that contains $D^X$ on $X$, we have that the direct image functor $f_\#$ is strict on both  $\mathcal{M}_{*D^X}$ and $\mathcal{M}_{!D^X}$, and
$$\mathcal{H}^if_\#\left( \mathcal{M}_{*D^X}\right)=\left(\mathcal{H}^i f_+\mathcal{M}\right)_{*D^Y},$$
$$\mathcal{H}^if_\# \left(\mathcal{M}_{!D^X}\right)=\left(\mathcal{H}^i f_+\mathcal{M}\right)_{!D^Y}.$$
\end{prop}
\begin{prop}\label{P:direct}
Given $f:\left(X,D^X\right)\to \left(Y,D^Y\right)$ as above, denote $E^X=D^X-f^{-1}D^Y$. We have both $f_\# \widetilde{\omega}_X\left(E^X\right)$ and $f_\# \widetilde{\omega}_X\left(D^X-E^X\right)$ are strict, where $\widetilde{\omega}\left(E^X\right)$ and $\widetilde{\omega}_X\left(D^X-E^X\right)$ are strict $\widetilde{\mathscr{D}}_{\left(X,D^X\right)}$-modules that induced by the trivial filtration. Further, we have
\begin{align*}
\mathcal{H}^if_\#\left( \widetilde{\omega}_X\left(E^X\right)\right)&\simeq
\left(\mathcal{H}^if_+\widetilde{\omega}_{X}[*D^X]\right)_{!D^Y},\\
\mathcal{H}^if_\# \left(\widetilde{\omega}_X\left(D^X-E^X\right)\right)&\simeq
\left(\mathcal{H}^if_+\widetilde{\omega}_{X}[!D^X]\right)_{*D^Y}.
\end{align*}
\end{prop}
\begin{proof}
We just show the first identity here, the second one follows similarly.

Decompose $f:\left(X,D^X\right)\to \left(Y,D^Y\right)$ into two morphisms of log pairs: $id:\left(X,D^X\right)\to \left(X, f^{-1}D^Y\right)$, and $f':\left(X, f^{-1}D^Y\right)\to \left(Y,D^Y\right)$. 
Now, by the definition of push forward functor, 
$$id_\# \widetilde{\omega}_X\left(E^X\right)\simeq \widetilde{\omega}_X\left(E^X\right)\otimes^{\mathbf{L}}_{\widetilde{\mathscr{D}}_{\left(X,D^X\right)}}\widetilde{\mathscr{D}}_{\left(X,f^{-1}D^Y\right)}.$$
Note that 
$$\widetilde{\omega}_X\left(E^X\right)=\mathbf{V}^{E^X}_\mathbf{0}\mathbf{V}^{f^{-1}D^Y}_{<\mathbf{0}}\widetilde{\omega}_X[*D^X].$$
Apply the comparison theorem with normal crossing boundary divisor in \cite{W17a} inductively, we have
\begin{align*}
&\mathbf{V}^{E^X}_\mathbf{0}\mathbf{V}^{f^{-1}D^Y}_{<\mathbf{0}}\widetilde{\omega}_X[*D^X]\otimes^{\mathbf{L}}_{\widetilde{\mathscr{D}}_{\left(X,D^X\right)}}\widetilde{\mathscr{D}}_{\left(X,f^{-1}D^Y\right)},\\
\simeq& \mathbf{V}^{f^{-1}D^X}_\mathbf{<0}\widetilde{\omega}_X[*D^X]\\
=& \widetilde{\omega}_X[*D^X]_{!f^{-1}D^Y}.
\end{align*}
Hence we are left to compute 
$f'_\# \widetilde{\omega}_X[*D^X]_{!f^{-1}D^Y},$ but this follows directly from the previous proposition.
\end{proof}
The following proposition can be viewed as a log version of Laumon's formula.
\begin{prop}[\cite{W17a}]\label{P:laumon}
Given $f:\left(X,D^X\right)\to \left(Y,D^Y\right)$ as above, let $\mathcal{M}$ be a strict $\widetilde{\mathscr{D}}_{\left(X,D^X\right)}$-module. Assume that $f_\#\mathcal{M}$ is strict. Then we have
$$\mathcal{H}^i f_{\widetilde{\#}}\text{Gr}^F \mathcal{M}:=R^i f_*\left(\text{Gr}^F \mathcal{M}\otimes^{\mathbf{L}}_{\mathfrak{A}_{\left(X,D^X\right)}}f^*\mathfrak{A}_{\left(Y,D^Y\right)}\right)=\text{Gr}^F \mathcal{H}^i f_\# \mathcal{M}.$$
\end{prop}
For the dual functor, we have the following.
\begin{prop}[\cite{W17a}]\label{P:dual}
Given a mixed Hodge module $\mathcal{M}$ on a log smooth pair $\left(X,H\right)$ with $H$ being smooth, we have that 
\begin{align*}
\mathbf{D}_{\left(X,H\right)} \mathcal{M}_{*H}&=\mathcal{M}'_{!H},\\
\mathbf{D}_{\left(X,H\right)} \mathcal{M}_{!H}&=\mathcal{M}'_{*H},
\end{align*}
where $\mathcal{M}'=\mathbf{D}_X\mathcal{M}$, the dual mixed Hodge module.
\end{prop}
The following proposition is a log version of \cite[Theorem 2.3]{PS13}.
\begin{prop}[\cite{W17a}]\label{P:dual gr}
Under the same condition as in the previous proposition, we denote by $\mathcal{G}(\mathcal{M}_{*H})$ and $\mathcal{G}(\mathcal{M}_{!H})$ the associated graded coherent $\mathscr{O}_{T^*_{(X,H)}}$-module of $\mathcal{M}_{*H}$ and $\mathcal{M}_{!H}$ respectively, supported on $T^*_{(X,H)}$, the space of log-cotangent bundle of $(X,H)$. Then we have 
\begin{align*}
 \mathcal{G}\left(\mathcal{M}'_{*H}\right)\simeq& \left(-1\right)^*_{T^*_{\left(X,H\right)}}
\mathbf{R}\mathcal{H}om_{\mathscr{O}_{T^*_{\left(X,H\right)}}}
\left(\mathcal{G}\left(\mathcal{M}_{!H}\right), p_X^*\omega_X[d_X] \otimes \mathscr{O}_{T^*_{\left(X,H\right)}}\right),\\
\mathcal{G}\left(\mathcal{M}'_{!H}\right)\simeq& \left(-1\right)^*_{T^*_{\left(X,H\right)}}
\mathbf{R}\mathcal{H}om_{\mathscr{O}_{T^*_{\left(X,H\right)}}}
\left(\mathcal{G}\left(\mathcal{M}_{*H}\right), p_X^*\omega_X[d_X] \otimes \mathscr{O}_{T^*_{\left(X,H\right)}}\right).
\end{align*} 
\end{prop}
Now we are ready to show some vanishing results that will be used to prove the main claim. The main vanishing result is the following theorem. Note that by taking $D=0$, it is Kodaira-Saito vanishing .
\begin{theorem}\label{T:more general saito V}
Fix a mixed Hodge module $\mathcal{M}$, a possibly un-reduced effective divisor $D$ and a semi-ample line bundle $\mathcal{L}$ on a smooth projective variety $X$. Assume further that $\mathscr{O}_X(D)\otimes \mathcal{L}$ is an ample line bundle, then we have the following vanishing:
\begin{align*}
\mathbb{H}^i\left(\text{Gr}^F_p \textsl{Sp}_X \mathcal{M}[*D]\otimes \mathcal{L}\right)=&0, \text{ for } i>0,\\
\mathbb{H}^i\left(\text{Gr}^F_p \textsl{Sp}_X \mathcal{M}[!D]\otimes \mathcal{L}^{-1}\right)=&0, \text{ for } i<0.
\end{align*}
\end{theorem}
\begin{proof}
The proof is similar to the proof of Kodaira-Saito vanishing in \cite[2.33. Proposition]{Sa90}. See also \cite{Po16}.

Since both taking (dual) localization and $\text{Gr}^F_p \textsl{Sp}_X$ are exact, by a standard reduction, we only need to show the case that $\mathcal{M}$ is a pure Hodge module with strict support $Z\subset X$. If $Z\subset D$, the vanishings are trivial. Further, by Grothendieck-Serre duality and its compatibility with the dual functor in mixed Hodge modules, it suffices to show the second vanishing.

Since $\mathcal{L}^m$ is globally generated for some integer $m$, by Bertini's theorem, we can find a global section $s\in H^0\left(X, \mathcal{L}^m\right)$ which defines a smooth hypersurface $Y$ that is non-characteristic for both $\mathcal{M}$ and $\mathcal{M}[!D]$. We have a finite covering:
$$\pi: \overline{X}:=\text{Spec}_X\left(\oplus_{0\leq i<m}\mathcal{L}^{-i}\right)\to X,$$
ramified along $Y$. Denote $U=X\setminus Y$ and $j: V=X\setminus (D+Y)\to X$ be the open embedding. Note that $V$ is affine by the ampleness assumption.  Because $\pi$ is non-characteristic by the construction respect to $\mathcal{M}$, the Hodge module $\pi_*\pi^*\mathcal{M}$ is well defined. Further, we have a natural injection $\mathcal{M}\to \pi_*\pi^*\mathcal{M}.$ Denote the Hodge module
$$\overline{\mathcal{M}}=\text{Coker}\left(\mathcal{M}\to \pi_*\pi^*\mathcal{M}\right).$$
Take the dual localization at $D$, we have the following short exact sequence:
$$0\to \mathcal{M}[!D]\to \pi_*\pi^*\mathcal{M}[!D]\to \overline{\mathcal{M}}[!D]\to 0.$$
By the construction, $\overline{\mathcal{M}}$ is the unique extension of $\overline{\mathcal{M}}\big|_U$ onto $X$ with strict support, since the eigenspace with eigenvalue $0$ of the monodromy operator on $\psi_Y\text{Sp}_X\overline{\mathcal{M}}$, the nearby cycle of $\text{Sp}_X\overline{\mathcal{M}}$ respect to $Y$, is empty. In particular, we have 
$$\overline{\mathcal{M}}[!D]=\overline{\mathcal{M}}[!(D+Y)].$$
Hence by Artin's vanishing, we have 
$$\mathbb{H}^i\left(\textsl{Sp}_X\overline{\mathcal{M}}[!D]\right)=\mathbb{H}^i_c\left(\textsl{Sp}_V\overline{\mathcal{M}}\big|_V\right)=0, \text{ for }i< 0.$$
Further, by the strictness of $a_+ \mathcal{M}[!D]$, where $a: X\to \text{pt}$, $\mathbb{H}^i\left(\text{Gr}^F_p \textsl{Sp}_X\overline{\mathcal{M}}[!D]\right)$ is a sub-quotient of $\mathbb{H}^i\left(\textsl{Sp}_X\overline{\mathcal{M}}[!D]\right)$. Hence 
\begin{equation}\label{E:vanishing already got}
\mathbb{H}^i\left(\text{Gr}^F_p \textsl{Sp}_X\overline{\mathcal{M}}[!D]\right)=0,\text{ for }i< 0.
\end{equation}

On the other hand, we have the short exact sequence 
$$0\to \mathcal{M}[!D]\to \left(\mathcal{M}[!D]\right)[*Y]\to i_* \mathcal{H}^1i^!\left(\mathcal{M}[!D]\right)\to 0,$$
with $\mathcal{H}^1i^!\mathcal{M}[!D]$ being a mixed Hodge module by the non-charactericity of $i$, where $i:Y\to X$ is the close embedding. We have
$$i_*\mathcal{H}^1i^!\left(\mathcal{M}\left[!D\right]\right)= \left(i_*\mathcal{H}^1i^!\mathcal{M}\right)\left[!D\right],$$
which can be checked at the level of perverse sheaf by \cite[2.11. Proposition]{Sa90}. 
Note that $i_*\mathcal{H}^1i^!\mathcal{M}$ has support on $Z\cap Y$. By induction on dimension, now we only need to show 
\begin{equation}\label{E:vanishing to get}
\mathbb{H}^i\left(\text{Gr}^F_p\textsl{Sp}_X\left(\left(\mathcal{M}[!D]\right)[*Y]\right)\otimes \mathcal{L}^{-1}\right)=0, \text{ for }i< 0.
\end{equation}
However, by replacing $\widetilde{M}$ in \cite[(2.33.2)]{Sa90} by $\overline{\mathcal{M}}[!D]$, (see also \cite[(8.8)]{Po16},) we have 
\begin{equation}\label{E:relation cyclic covering}
{Gr}^F_p \left(\overline{\mathcal{M}}[!D]\right)\simeq \text{Gr}^F_p\left(\left(\mathcal{M}[!D]\right)[*Y]\right)\otimes \overline{\mathcal{L}},
\end{equation}
where 
$$\overline{\mathcal{L}}:=\text{Coker}\left(\mathscr{O}_X\to \pi_*\mathscr{O}_{\overline{X}}\right)\simeq \oplus_{0<i<m}\mathcal{L}^{-i}.$$
Now (\ref{E:vanishing to get}) follows by (\ref{E:vanishing already got}) and (\ref{E:relation cyclic covering}), which concludes the proof.
\end{proof}

\begin{coro}\label{coro:vanishing of gr}
Fix $\left(P^{r,d},L\right)$, $p: P^{r,d}\to A^d$ as in Theorem \ref{T:over quasi-abelian variety}. Fix a mixed Hodge module $\mathcal{M}$ on $P^{r,d}$ and an ample line bundle $\mathcal{A}$ over $A^d$. Assume that $r=0$ or $1$, or $\mathcal{M}$ is of normal crossing type respect to a divisor $L'$ that contains $L$. Then we have 
\begin{align*}
H^i\left(P^{r,d}, \text{Gr}^F_p \mathcal{M}_{*L}\otimes p^*\mathcal{A}\right)=0, \text{ for } i>0,\\
H^i\left(P^{r,d}, \text{Gr}^F_p \mathcal{M}_{!L}\otimes p^*\mathcal{A}^{-1}\right)=0, \text{ for } i<0,
\end{align*}
for any integer $k$.
\end{coro}
\begin{proof}
Note that if $r=0$, it is just \cite[Lemma 2.5]{PS13}. We use $P$ and $A$ to replace $P^{r,d}$ and $A^d$ to simplify the notations. We only show the vanishing of the $i>0$ case here, the other case follows similarly. Let $m=r+d=\dim P$.
We have that
\begin{align*}
\text{Gr}^F_p\textsl{Sp}_{\left(P,L\right)}\mathcal{M}_{*L}=[&\text{Gr}^F_p\mathcal{M}_{*L}\otimes \wedge^m\mathcal{T}_{\left(P,L\right)}\\
&\to \text{Gr}^F_{p+1}\mathcal{M}_{*L}\otimes \wedge^{m-1}\mathcal{T}_{\left(P,L\right)} \to ... \to \text{Gr}^F_{p+m} \mathcal{M}_{*L}],
\end{align*}
placed in cohomological degree $-m,...,0$. Note that $\mathscr{O}_P(L)\otimes p^*\mathcal{A}$ is ample and $ p^*\mathcal{A}$ is semi-ample. According to the previous vanishing theorem and Proposition \ref{P:log rep}, for $i>0$, we have that 
\begin{equation}\label{e:vanishing of gr dr}
\mathbb{H}^i\left(\text{Gr}^F_p\textsl{Sp}_{\left(P,L\right)}\mathcal{M}_{*L}\otimes p^*\mathcal{A}\right)=0.
\end{equation}
We have that $\mathcal{T}_{\left(P,L\right)}\simeq \mathscr{O}_P^{\oplus m}$.  Since $F_p\mathcal{M}_{*L}=0$, for $p\ll 0$, take $p+m$ be the smallest integer such that $\text{Gr}^F_{p+m}\mathcal{M}_{*L}$ is not trivial. Then by (\ref{e:vanishing of gr dr}), we get that $H^i\left(P, \text{Gr}^F_{p+m} \mathcal{M}_{*L}\otimes p^*\mathcal{A}\right)=0$, for $i>0$. By induction, we can conclude the vanishing of $\text{Gr}^F_p \mathcal{M}_{*L}$ for any $p$.
\end{proof}
\begin{coro}\label{coro:vanishing}
Assume that we have a morphism of log smooth pairs $f:\left(X,D\right)\to \left(P^{r,d},L\right)$, with $r=0$ or $1$ as in Theorem \ref{T:over quasi-abelian variety}, with $p:P^{r,d}\to A^d$, the natural projection. Then we have 
$$H^i\left(P^{r,d}, \mathcal{H}^k f_{\widetilde{\#}}\omega_X\left(f^{-1}L\right)\otimes p^*\mathcal{A}\right)=0,$$
 for any ample line bundle $\mathcal{A}$ over $A^d$, all $k\in\mathbb{Z}$ and $i>0$.
\end{coro}
\begin{proof}
By Proposition \ref{P:direct} and Proposition \ref{P:laumon}, we have 
$$\mathcal{H}^k f_{\widetilde{\#}}\omega_X\left(f^{-1}L\right)
\simeq \text{Gr}^F \left(\mathcal{H}^k f_\# \widetilde{\omega}_X\left(f^{-1}L\right)\right)
\simeq \text{Gr}^F \left(\left(\mathcal{H}^k f_+\left(\widetilde{\omega}_X[!D]\right)\right)_{*L}\right).$$
Now it follows by the previous Corollary. 
\end{proof}

\section{Connection to mixed Hodge modules and the proof of the claim}
\begin{prop}\label{P:gr direct}
Let $g:\left(Y,D^Y\right)\to \left(P^{r,d},L\right)$ be a morphism between two log-smooth pairs, with $r=0$ or $1$, as in Theorem \ref{T:over quasi-abelian variety}. Denote $E^Y=D^Y-g^{-1}L$. We have
$$g_{\widetilde{\#}}\omega_Y\left(E^Y\right) \simeq \mathbf{R}g_*\left(C_{Y,D^Y,\bullet}\otimes \mathscr{O}_Y\left(E^Y-D^Y\right)\right),$$
and dually
$$g_{\widetilde{\#}}\omega_Y\left(D^Y-E^Y\right) \simeq \mathbf{R}g_*\left(C_{Y,D^Y,\bullet}\otimes \mathscr{O}_Y\left(-E^Y\right)\right),$$
where $C_{Y,D^Y,\bullet}$ is defined in the proof of Theorem \ref{T:over quasi-abelian variety}.
In particular, we have 
$$\text{Gr}^F\left(\left(\mathcal{H}^i g_+\widetilde{\omega}_{Y}[*D^Y]\right)_{!L} \right) \simeq R^ig_*\left(C_{Y,D^Y,\bullet}\otimes \mathscr{O}_Y\left(E^Y-D^Y\right)\right),$$
and 
$$\text{Gr}^F\left(\left(\mathcal{H}^i g_+\widetilde{\omega}_{Y}[!D^Y]\right)_{*L}\right) \simeq R^i g_*\left(C_{Y,D^Y,\bullet}\otimes \mathscr{O}_Y\left(-E^Y\right)\right).$$
\end{prop}
\begin{proof}
This follows exactly as \cite[2.11]{PS13} by considering Proposition \ref{P:direct} and Proposition \ref{P:laumon}. More precisely, we have 
\begin{align*}
&g_{\widetilde{\#}}\omega_Y\left(E^Y\right)\\
\simeq& \mathbf{R} g_*\left(\omega_Y\left(E^Y\right)\otimes^{\mathbf{L}}_{\mathfrak{A}_{\left(Y,D^Y\right)}}g^*\mathfrak{A}_{\left(P^{r,d},L\right)}\right)\\
\simeq&\mathbf{R} g_*\left(\left[\omega_Y\left(E^Y\right)\otimes_{\mathscr{O}_Y} \wedge^{-\bullet} \mathcal{T}_{\left(Y,D^Y\right)}\otimes_{\mathscr{O}_Y} \mathfrak{A}_{\left(Y,D^Y\right)}\right]\otimes^{\mathbf{L}}_{\mathfrak{A}_{\left(Y,D^Y\right)}}g^*\mathfrak{A}_{{\left(P^{r,d},L\right)}}\right)\\
\simeq&  \mathbf{R}g_*\left(\left[\mathscr{O}_Y\left(E^Y-D^Y\right)\otimes_{\mathscr{O}_Y} \Omega^{n+\bullet}_Y\left(\log   D^Y \right)\otimes_{\mathscr{O}_Y} \mathfrak{A}_{\left(Y,D^Y\right)}\right]\otimes^{\mathbf{L}}_{\mathfrak{A}_{\left(Y,D^Y\right)}}g^*\mathfrak{A}_{{\left(P^{r,d},L\right)}}\right)\\
\simeq&\mathbf{R}g_*\left(C_{Y,D^Y,\bullet}\otimes \mathscr{O}_Y\left(E^Y-D^Y\right)\right),
\end{align*}
where the second identity is due to the canonical resolution \cite[(1)]{W17a} and the third identity is due to \cite[Exercise II 5.16 (b)]{Ha77}.

The second identity of the proposition can be shown similarly.
\end{proof}

\begin{prop} \label{P:v}
Notations as in the previous proposition, let $p:P^{r,d}\to A^d$ be the natural projection, and $\mathcal{A}$ be an ample line bundle over $A^d$. Then we have that for any $i$ and $l>0$,
$$H^l\left(X, R^ig_*\left(C_{Y,D^Y,\bullet}\otimes\mathscr{O}_Y\left(-E^Y\right)\otimes p^*\mathcal{A}\right)\right)=0.$$
\end{prop}
\begin{proof}
It follows by the previous proposition and Proposition \ref{coro:vanishing}. 
\end{proof}

\begin{proof}[Proof of Claim \ref{C:main claim} (Main Claim)]
Denote 
$$\mathcal{E}=\mathbf{R}p_{V*} R^0\hat{g}_*\left(\hat{\mathcal{L}}\otimes p_Y^*\mathscr{O}_Y\left(-D^Y\right)\otimes C_{Y,D^Y}\right).$$ We first show that for $l>0$,
\begin{equation}\label{e:ve}
\mathcal{H}^l \mathcal{E}=0.
\end{equation}
In particular, $\mathcal{E}$ is a sheaf.

Since both $p_P$ and $p_Y$ are affine and by Proposition \ref{P:v}, we have 
\begin{align*}
&H^l\left(P\times V, R^0\hat{g}_*\left(\hat{\mathcal{L}}\otimes p_Y^*\mathscr{O}_Y\left(-D^Y\right)\otimes C_{Y,D^Y}\right)\right) \\
\simeq&H^l\left(P, p_{1*} R^0\hat{g}_*\left(\hat{\mathcal{L}}\otimes p_Y^*\mathscr{O}_Y\left(-D^Y\right)\otimes C_{Y,D^Y}\right)\right)\\
\simeq&H^l\left(P,  R^0g_* p_{Y*}\left(\hat{\mathcal{L}}\otimes p_Y^*\mathscr{O}_Y\left(-D^Y\right)\otimes C_{Y,D^Y}\right)\right) \\
\simeq&H^l\left(P,  R^0g_* \left(C_{Y,D^Y,\bullet}\otimes \mathscr{O}_Y\left(-E^Y\right)\right)\otimes p^*\mathcal{A}\right) \\
=&0,
\end{align*}
where $E^Y=D^Y-f^{-1}L$.

Hence, due to the degeneration of the Leray spectral sequence induced by $p_{V*}$, we have 
\begin{align*}
&H^0\left(V, \mathcal{H}^l \mathcal{E}\right)\\
=&H^0\left(V, R^l p_{V*} R^0\hat{g}_*\left(\hat{\mathcal{L}}\otimes p_Y^*\mathscr{O}_Y\left(-D^Y\right)\otimes C_{Y,D^Y}\right)\right)\\
\simeq&H^l\left(P\times V, R^0\hat{g}_*\left(\hat{\mathcal{L}}\otimes p_Y^*\mathscr{O}_Y\left(-D^Y\right)\otimes C_{Y,D^Y}\right)\right)\\
=&0,
\end{align*} 
and so that (\ref{e:ve}) follows.

To prove that
$p_{V*}R^0\hat{g}_*\left(\hat{\mathcal{L}}^{-1}\otimes C_{Y,D^Y}\right)$ is torsion free, now it suffices to show that 
\begin{equation}\label{e:claim dual}
p_{V*}R^0\hat{g}_*\left(\hat{\mathcal{L}}^{-1}\otimes C_{Y,D^Y}\right)=\mathbf{R}^0\mathcal{H}om\left(\mathcal{E}, \mathscr{O}_V\right).
\end{equation}

By Grothendieck Duality, we have that 
\begin{align}
&\mathbf{R}\mathcal{H}om\left(\mathcal{E},\mathscr{O}_V\right)\nonumber\\
\simeq&\mathbb{D}_V\left(\mathbf{R}p_{V*} R^0\hat{g}_*\left(\hat{\mathcal{L}}\otimes p_Y^*\mathscr{O}_Y\left(-D^Y\right)\otimes C_{Y,D^Y}\right)\right)[-m]\nonumber\\
\simeq&\mathbf{R}p_{V*} \mathbb{D}_{P\times V}\left( R^0\hat{g}_*\left(\hat{\mathcal{L}}\otimes p_Y^*\mathscr{O}_Y\left(-D^Y\right)\otimes C_{Y,D^Y}\right)\right)[-m].\label{E:rhom}
\end{align}
Note that by definition, 
$$\hat{\mathcal{L}}\otimes p_Y^*\mathscr{O}_Y\left(-D^Y\right)=p_Y^*\left(\mathscr{O}_Y\left(-E^Y\right)\otimes g^*p^*\mathcal{A}\right) .$$
By Proposition \ref{P:gr direct} and $T^*_{(P,L)}=P\times V$, we have 
\begin{align*}
R^0\hat{g}_*\left(C_{Y,D^Y}\otimes p_Y^* \mathscr{O}_Y\left(E^Y-D^Y\right)\right)\simeq& \mathcal{G}\left(\left(\mathcal{H}^0 g_+ \widetilde{\omega}_Y[*D^Y]\right)_{!L}\right),\\
R^0\hat{g}_*\left(C_{Y,D^Y}\otimes p_Y^* \mathscr{O}_Y\left(-E^Y\right)\right)\simeq& \mathcal{G}\left(\left(\mathcal{H}^0 g_+ \widetilde{\omega}_Y[!D^Y]\right)_{*L}\right).
\end{align*}
Since up to a Tate twist, we have 
$$\mathbf{D}_P \left(\mathcal{H}^0 g_+ \widetilde{\omega}_Y[!D^Y]\right)=\mathcal{H}^0 g_+ \widetilde{\omega}_Y[*D^Y],$$
by Proposition \ref{P:dual} and Proposition \ref{P:dual gr}, we have 
\begin{align*}
&\mathbf{R}\mathcal{H}om\left(R^0\hat{g}_*\left(C_{Y,D^Y}\otimes p_Y^* \mathscr{O}_Y\left(-E^Y\right)\right), p_P^*\omega_P[m]\right)\\
\simeq & \mathcal{G}\left(\mathbf{D}_{(P,L)}\left(\mathcal{H}^0 g_+ \widetilde{\omega}_Y[!D^Y]\right)_{*L}\right)\\
\simeq & \mathcal{G}\left(\left(\mathcal{H}^0 g_+ \widetilde{\omega}_Y [*D^Y]\right)_{!L} \right)\\
\simeq &R^0\hat{g}_*\left(C_{Y,D^Y}\otimes p_Y^* \mathscr{O}_Y\left(E^Y-D^Y\right)\right).
\end{align*}
Comparing it with (\ref{E:rhom}), and by definition
$$\hat{\mathcal{L}}^{-1}=p_Y^*\left(\mathscr{O}_Y\left(E^Y-D^Y\right)\otimes g^*p^*\mathcal{A}^{-1}\right),$$
we obtain
$$\mathbf{R}\mathcal{H}om\left(\mathcal{E}, \mathscr{O}_V\right)\simeq \mathbf{R}p_{V*}R^0\hat{g}_*\left(\hat{\mathcal{L}}^{-1}\otimes C_{Y,D^Y}\right),
$$
which implies (\ref{e:claim dual}).
\end{proof}
\bibliographystyle{alpha}
\bibliography{mybib}
\end{document}